\documentclass[letterpaper, 10 pt, conference]{ieeeconf}  
\IEEEoverridecommandlockouts                                                                    
\pdfoutput=1

\usepackage{amsthm}

\usepackage[utf8]{inputenc}
\usepackage{graphicx}
\usepackage{amsmath} 
\usepackage{amsfonts}
\usepackage{float}
\usepackage{amssymb}
\usepackage{xcolor,soul}
\usepackage{hyperref}
\usepackage{mathtools}
\usepackage{nccmath}
\usepackage[noadjust]{cite}

\theoremstyle{plain}
	\newtheorem{theorem}{Theorem}
	\newtheorem{corollary}{Corollary}
	\newtheorem{lemma}{Lemma}
	
	\newtheorem{assumption}{Assumption}
\theoremstyle{definition}
	\newtheorem{definition}{Definition}

\title{\LARGE \bf Data-Driven Internal Model Control of Second-Order Discrete Volterra Systems}

\author{Juan G. Rueda-Escobedo$^{1}$ and Johannes Schiffer$^{1}$
\thanks{$^{1}$J.~G.~Rueda-Escobedo and J.~Schiffer are with the Fachgebiet Regelungssysteme und Netzleittechnik, Brandenburgische Technische Universität Cottbus - Senftenberg {\tt\small\{ruedaesc,schiffer\}@b-tu.de}.}%
}

\begin{document}
	
\maketitle
\thispagestyle{empty}
\pagestyle{empty}

\begin{abstract}
The increase in system complexity paired with a growing availability of operational data
has motivated a change in the traditional control design paradigm. Instead of modeling the system by first principles and then proceeding with a (model-based) control design, the data-driven control paradigm proposes to directly characterize the controller from data. By exploiting a fundamental result of Willems and collaborators, this approach has been successfully applied to linear systems, yielding data-based formulas for many classical linear controllers.
In the present paper, the data-driven approach is extended to a class of nonlinear systems, namely second-order discrete Volterra systems. Two main contributions are made
for this class of systems. At first, we show that - under a necessary and sufficient condition on the input data excitation - a data-based system representation can be derived from input-output data and used to replace an explicit system model.
%
That is, the fundamental result of Willems {\em et al.} is extended to this class of systems. Subsequently a data-driven internal model control formula for output-tracking is derived. The approach is illustrated via two simulation examples.
\end{abstract}

\section{INTRODUCTION}
	
	The steady drive for digitization of engineering processes has greatly increased the amount of available system input-output data 
	\cite{HouWang2013,YinDingXie2014,LamAnnEngel2017}. 
	At the same time, the system complexity has also increased \cite{HouWang2013,LamAnnEngel2017}, hindering the model derivation from first principles. This has motivated a change in the manner in which the control design is approached \cite{HouWang2013}. Instead of modeling the system by using first principles and then proceeding to the control design, the new paradigm is to take advantage of the available data to directly characterize a controller, whose structure is typically assumed known. This approach is referred to as \emph{data-driven control} \cite{DataDriven2012,HouWang2013}. 
	
	
	
	However, in order to provide standard closed-loop performance guarantees, such as stability and robustness, it is still necessary to assume a certain structure of the system to be controlled. 
	Clearly, a natural first approach is to assume that the underlying system dynamics are linear. In this respect, data-driven results for linear quadratic tracking \cite{MarkRap2007,Markovsky08}, dynamical feedback \cite{ParkIkeda2009}, state-feedback and optimal control \cite{PangBian2018,Persis2019,Tu2018,Rotulo2019,Berberich2019,Dean2019} as well as model predictive control \cite{Coulson19,BerKohMull2019} have been developed. The collection of these contributions practically cover a wide range of classical design techniques for linear systems. Therefore, to further progress in the consolidation of the data-driven control theory, it is necessary to take a step forward and start developing methods allowing to deal with nonlinear systems.
	
	This motivates the present note, which is devoted to the extension of the data-driven theory to nonlinear Discrete Volterra Systems (DVSs).
	DVSs haven been classically used to approximate the behavior of nonlinear systems and they play an important role in nonlinear system identification theory \cite{Janczak2004,Ogunfu2007,SchmNiag2014}. Many classical nonlinear systems, like Hammerstein and Wiener systems with smooth nonlinearities, can be described by DVSs \cite[Sec.~2]{Doyle2002}. Likewise, under mild assumptions on the output nonlinearity, Lur\'{e} type systems can be represented by DVSs \cite[Sec.~2]{Doyle2002}. Additionally, multilayer neural networks can be modeled as the interconnection of several Wiener systems and can thus be approximated by high-order DVSs \cite{GovRam1990,DavGas1991,MarmZhao1997,Jancz2009}. In the area of process control, DVSs have shown to be of great usefulness for modeling and control design \cite{ManerDoyle1996,LingRiv1998,VarAll2004,MedNie2010} and also be employed to represent a class of communication systems \cite{AntHesSil2012}.	
	Very recently, a trajectory-based approach to represent Hammerstein and Wiener systems has been proposed in the context of data-driven control \cite{BerAll2019} with future applications in model predictive control. Hence, there already is an interest in extending the data-driven theory to systems contained in the class of DVSs. 
	
	
	An outstanding subclass within the broad class of DVSs are second-order DVSs. This (sub)class of systems has a strong connection with the bispectrum \cite{NikPet1993}, which helps to characterize stochastic signals beyond Gaussian sequences, 
	and has applications, for example, in amplitude and frequency modulation \cite{MargKaiQua1993}. Another important application for second-order DVSs is the approximation of the input-output response of bilinear systems when using explicit Runge-Kutta formulas for the discretization \cite{ManerDoyle1996},\cite[Sec.~4.7]{Doyle2002}. Finally, working with this (sub)class of systems represents a ``first step beyond linearity'' while keeping a reasonable dimensionality and number of parameters \cite[Sec.~3]{Doyle2002}.
	
	Many of the abovementioned results on data-driven control for linear systems exploit a fundamental lemma derived by Willems and collaborators within the behavorial systems framework \cite{Willems2005}. The importance of Willems {\em et al.}'s result in a data-driven context is that it establishes a necessary and sufficient condition on the required excitation of the input data of a linear system, such that all its possible trajectories can be represented solely by means of the system's input-output data. 
	Building upon this fact, existing data-driven control approaches seek to directly parametrize the controller through data, see e.g. \cite{Persis2019}. 	
	As a consequence, in order to extend the data-driven control approach from linear to nonlinear systems similar tools are required.
	
	In this context, the present work contains the following two main contributions for data-driven control of DVSs:
	
	\begin{enumerate}
		\item	A data-based representation theory for second-order DVSs is developed, hence extending the fundamental result for linear systems \cite[Theo.~1]{Willems2005}, \cite[Lem.~2]{Persis2019} to this class of nonlinear systems. More precisely, under a necessary and sufficient condition on the required excitation of the input data, our result enables the replacement of an explicit system model by input-output data.
		\item A data-driven internal model control (IMC) formula for output-tracking in second-order DVSs is given. In general, even if the system dynamics are known, designing an output-tracking control for DVSs is challenging due to their nonlinear characteristics\footnote{DVSs are inherently bounded-input-bounded-output (BIBO) stable systems. Therefore, internal stability is guaranteed and thus only output-tracking is relevant for control \cite{Doyle1995,LingRiv1998,Doyle2002}.}. A popular class of controllers that can achieve the desired output tracking are IMCs \cite{Doyle1995,LingRiv1998,Doyle2002}. 
		Thus, our second main contribution is to derive an IMC formula, which is directly parametrized through data. This is accomplished by using the data-based representation results from 1).
	\end{enumerate}

	
	The paper organization is as follows. In Section \ref{Sec:Preliminaries} the notation used along the note and the class of systems under study are introduced. In Section \ref{Sec:DataBasedRep} the data-based representation theory for second-order DVSs is presented. A data-driven IMC formula for output-tracking in second-order DVSs are developed in Section \ref{Sec:DataDrivenCon}. Finally, in Section~\ref{Sec:Simulation}, the application of the methods developed in Section \ref{Sec:DataBasedRep} and Section \ref{Sec:DataDrivenCon} is illustrated via numerical simulations by controlling an ideal DVS and a bilinear system. 

\section{PRELIMINARIES}
\label{Sec:Preliminaries}
	\subsection{Notation}
		Along the note, $\mathbb{Z}$ represents the set of integer numbers and $\mathbb{R}$ the set of real numbers. Let $\mathbb{F}$ be either $\mathbb{Z}$ or $\mathbb{R}$. Then $\mathbb{F}_{> 0}$ ($\mathbb{F}_{\geq 0}$) denotes the set of all elements of $\mathbb{F}$ greater than (or equal to) zero. 
		$\mathbf{I}_n\in\mathbb{R}^{n\times n}$, with $n\in\mathbb{Z}_{> 0}$, denotes the $n\times n$ identity matrix. 
		The Kronecker product of two matrices $A\in\mathbb{R}^{n\times m}$ and $B\in\mathbb{R}^{r\times s}$ is denoted by $A\otimes B$.
		
		Let $A\in\mathbb{R}^{n\times n}$ be a symmetric matrix. We denote by $\mathrm{vech}(A)\in\mathbb{R}^{\frac{n}{2}(n+1)}$ the column vector containing the entries in the lower triangular part of $A$, i.e.,
		\begin{align*}
		\mathrm{vech}(A)=\left[\begin{array}{c c c c c}a_{(1,1)} & a_{(2,1)} & a_{(2,2)} & \cdots & a_{(n,n)}\end{array}\right]^\top.
		\end{align*}
		Given a signal $u:\mathbb{Z}\to\mathbb{R}^m$, the column vector $u_{[k,k+T]}\in\mathbb{R}^{m\cdot T}$, with $k\in\mathbb{Z}$ and $T\in\mathbb{Z}_{> 0}$, corresponds to
		\begin{align*}
		u_{[k,k+T]}=\left[\begin{array}{c c c c}u^\top(k) & u^\top(k+1) & \cdots & u^\top(k+T-1)\end{array}\right]^\top.
		\end{align*}
		Given a signal $u:\mathbb{Z}\to\mathbb{R}^m$, the Hankel matrix $U_{\{k,L,N\}}\in\mathbb{R}^{(m\cdot L)\times N}$ associated to $u(k)$, with $k\in\mathbb{Z}$, $L\in\mathbb{Z}_{\geq 0}$ and $N\in\mathbb{Z}_{> 0}$ is given by
		\begin{align}
		U&_{\{k,L,N\}}=\nonumber\\
		&\left[\begin{array}{c c c c}u(k) & u(k+1) & \cdots & u(k+N-1)\\ u(k+1) & u(k+1) & \cdots & u(k+N)\\ \vdots & \vdots & \ddots & \vdots\\ u(k+L-1) & u(k+L) & \cdots & u(k+L+N-2)\end{array}\right].\label{Eq:Not01}
		\end{align}
		Consider a vector $u\in\mathbb{R}^{m}$, $u=[u_1\,u_2\,\cdots\,u_m]^\top$. By $u^2\in\mathbb{R}^{\frac{m}{2}(m+1)}$ we denote the vector
		\begin{align}
		u^2=\mathrm{vech}(uu^\top)=\left[\begin{array}{c c c c c}u^2_1 & u_1u_2 & u^2_2 & \cdots & u^2_m\end{array}\right]^\top. \label{Eq:Not02}
		\end{align}
		For a matrix $A\in\mathbb{R}^{n\times m}$, $A^\dagger$ denotes the Moore-Penrose pseudoinverse of $A$. If A has either full-row rank or full-column rank, it follows that \cite[Prop.~6.1.5]{Bernstein2009}
		\begin{align*}
		A^\dagger =\left\{\begin{array}{ c c c}
		(A^\top A)^{-1}A^\top & \mathrm{if} & \mathrm{rank}(A)=m,\\
		A^\top(A A^\top)^{-1} & \mathrm{if} & \mathrm{rank}(A)=n.
		\end{array}\right.
		\end{align*}
		
	\subsection{Second-order discrete Volterra systems (DVSs)}
		In this note, we focus on data-driven control for \emph{second order $M$-dimensional discrete Volterra systems (DVSs)}. These systems are described by the input-output equation \cite{Rugh2002,Doyle2002}
		\begin{align}
		y(k)=\sum^{M}_{i=0}&\alpha_{i+1}u(k-i)\nonumber\\
		&+\sum^{M}_{i=0}\sum^{M}_{j=0}\beta_{(i+1,j+1)}u(k-i)u(k-j),\label{Eq:VolterraEq01}
		\end{align}
		with $k\in\mathbb{Z}$, $M\in\mathbb{Z}_{> 0}$ the memory length of the system, $u(k)\in\mathbb{R}$ the system input signal, $y(k)\in\mathbb{R}$ the system output signal and $\alpha_{i}\in\mathbb{R}$ and $\beta_{(i,j)}\in\mathbb{R}$ the system parameters, with $i,j=\{1,\,2,\cdots,M+1\}$. 
		Following standard practice \cite[Sec.~3.1]{Doyle2002}, we assume symmetricity of the system \emph{kernels} from which it follows that $\beta_{i,j}=\beta_{j,i}$.
		
		The system \eqref{Eq:VolterraEq01} can be understood as the combination of two operators applied to the input signal $u(k)$. Denote these operators by $\mathbf{P}_1$ and $\mathbf{P}_2$. Then, the system \eqref{Eq:VolterraEq01} can be equivalently described as
		\begin{align}
		y(k)=\mathbf{P}_1\big(u(k)\big)+\mathbf{P}_2\big(u(k)\big),\label{Eq:VolterraEq02}
		\end{align}
		with
		\begin{equation}
		\begin{split}
		\mathbf{P}_1\big(u(k)\big)&=\theta^\top_1\mu(k),\\
		\mathbf{P}_2\big(u(k)\big)&=\theta^\top_2\mu^2(k),\\
		\mu(k)&=\left[\begin{array}{c c c c}u(k) & u(k-1) & \cdots & u(k-M)\end{array}\right]^\top,\\
		\theta_1&=\left[\begin{array}{c c c c}\alpha_1 & \alpha_2 & \cdots & \alpha_{M+1}\end{array}\right]^\top,\\
		\theta_2&=\left[\begin{array}{c c c c}\beta_{(1,1)} & 2\beta_{(1,2)} & \cdots & \beta_{(M+1,M+1)}\end{array}\right]^\top,
		\end{split}\label{Eq:OpDef}
		\end{equation}
		where the vector $\mu(k)$ contains the input $u(k)$ and its first $M$ delays and $\mu^2(k)$ follows the convention introduced in \eqref{Eq:Not02}. As can be seen from \eqref{Eq:OpDef}, $\mathbf{P}_1$ is linear with respect to the input $u(k)$ while $\mathbf{P}_2$ it is not. Furthermore, $\mathbf{P}_1$ is a causal and stable operator.
		

\section{DATA-BASED SYSTEM REPRESENTATION OF DVSs}
\label{Sec:DataBasedRep}
	In this section we provide an input-output data-based representation of the second-order DVS \eqref{Eq:VolterraEq02}. The proofs of all results contained in this section are given in the Appendix. 
	
	For the presentation of the results and in analogy to the linear case, see e.g. \cite[Def.~1]{Persis2019}, it is convenient to introduce a persistence of excitation concept \emph{ad hoc} to the class of systems under study, i.e., the DVS \eqref{Eq:VolterraEq02}. To this end, consider an observed input sequence $u_{ob,[-M,T]}$ and its associated sequences $\mu_{ob,[0,T]}$ and $\mu^2_{ob,[0,T]}$ where, for $k\in[0,T-1]$, their elements have to be understood as in \eqref{Eq:OpDef} and \eqref{Eq:Not02}. Consider the Hankel matrices corresponding to the associated sequences $\mu_{ob}(k)$ and $\mu^2_{ob}(k)$, i.e., 
	\begin{align}
	M&_{ob,\{0,L,T-L+1\}}=\nonumber\\
	&\left[\begin{array}{c c c c}	\mu_{ob}(0) & \mu_{ob}(1) & \cdots & \mu_{ob}(T-L)\\ 
	\mu_{ob}(1) & \mu_{ob}(2) & \cdots & \mu_{ob}(T-L+1)\\
	\vdots & \vdots & \ddots & \vdots\\
	\mu_{ob}(L-1) & \mu_{ob}(L) & \cdots & \mu_{ob}(T-1)\end{array}\right],\label{Eq:HankelM}\\
	M&^2_{ob,\{0,L,T-L+1\}}=\nonumber\\
	&\left[\begin{array}{c c c c}	\mu^2_{ob}(0) & \mu^2_{ob}(1) & \cdots & \mu^2_{ob}(T-L)\\ 
	\mu^2_{ob}(1) & \mu^2_{ob}(2) & \cdots & \mu^2_{ob}(T-L+1)\\
	\vdots & \vdots & \ddots & \vdots\\
	\mu^2_{ob}(L-1) & \mu^2_{ob}(L) & \cdots & \mu^2_{ob}(T-1)\end{array}\right],\label{Eq:HankelM2}
	\end{align}
	as well as the block matrix
	\begin{align}
	\mathcal{M}_{ob,\{0,L,T-L+1\}}=\left[\begin{array}{c}M_{ob,\{0,L,T-L+1\}}\\ M^2_{ob,\{0,L,T-L+1\}}\end{array}\right].\label{Eq:HankelMcal}
	\end{align}
	We introduce the following notion, cf. \cite[Def.~1]{Persis2019}.	
	\begin{definition}
		\label{Def:PE}
		Consider the sequence $u_{ob,[-M,T]}$ and associated sequences $\mu_{ob,[0,T]}$ and $\mu^2_{ob,[0,T]}$ built from $u_{ob,[-M,T]}$ following \eqref{Eq:OpDef} and \eqref{Eq:Not02}. The sequence $u_{ob,[-M,T]}$ is persistently exciting of order $L>0$ w.r.t. the system \eqref{Eq:VolterraEq02} if the matrix $\mathcal{M}_{ob,\{0,L,T-L+1\}}$ given in \eqref{Eq:HankelMcal} has full-row rank, i.e., if it has rank equal to $\frac{1}{2}L(M+1)(M+4)$.
	\end{definition}
	
	Now we are in the position of introducing the main data-based representation result of this note, which is a natural counterpart to the representation theorems for linear systems reported in \cite[Lem.~2]{Persis2019}, \cite[Theo.~1]{Willems2005}.
	
	\begin{theorem}[\textbf{Data-based trajectory representation of a second-order DVS}]
		\label{Theo:Representation}
		Consider the system \eqref{Eq:VolterraEq02} and observed input-output signals $u_{ob}(k)$ and $y_{ob}(k)$. Construct the $T$-long sequences $\mu_{ob,[0,T]}$, $\mu^2_{ob,[0,T]}$ and $y_{ob,[0,T]}$ according to \eqref{Eq:OpDef} and \eqref{Eq:Not02}. In addition, follow \eqref{Eq:HankelM} and \eqref{Eq:HankelM2}, to build the associated Hankel matrices $M_{ob,\{0,L,T-L+1\}}$, $M^2_{ob,\{0,L,T-L+1\}}$ and $Y_{ob,\{0,L,T-L+1\}}$, with $T\gg L>0$.
		\begin{enumerate}
			\item[i)]	Iff $u_{ob,[-M,T]}$ is persistently exciting of order $L$ w.r.t. the system \eqref{Eq:VolterraEq02}, then for every $L$-long sequences $\mu_{[0,L]}$, $\mu^2_{[0,L]}$ and $y_{[0,L]}$ associated to the system \eqref{Eq:VolterraEq02}, there exists $g\in\mathbb{R}^{T-L+1},$ such that
			\begin{align}
			\left[\begin{array}{c}\mu_{[0,L]}\\ \mu^2_{[0,L]}\\ y_{[0,L]}\end{array}\right]=\left[\begin{array}{c}M_{ob,\{0,L,T-L+1\}}\\ M^2_{ob,\{0,L,T-L+1\}}\\ Y_{ob,\{0,L,T-L+1\}}\end{array}\right]g. 
			\end{align}
			\item[ii)]	Consider an arbitrary sequence $v:[-M,L-1]\to\mathbb{R}$ and corresponding vectors $\nu(\tau)=[v(\tau)\,v(\tau-1)\,\cdots\,v(\tau-M)]^\top$ and $\nu^2(\tau)$ according to \eqref{Eq:Not02}, for $\tau=\{0,1,\cdots,L-1\}$. Iff $u_{ob,[-M,T]}$ is persistently exciting of order $L$ w.r.t. the system \eqref{Eq:VolterraEq02}, then the $L$-long sequences $\mu_{[0,L]}$, $\mu^2_{[0,L]}$ and $y_{[0,L]}$ generated by
			\begin{align}
			\left[\begin{array}{c}\mu_{[0,L]}\\ \mu^2_{[0,L]}\\ y_{[0,L]}\end{array}\right]=\left[\begin{array}{c}M_{ob,\{0,L,T-L+1\}}\\ M^2_{ob,\{0,L,T-L+1\}}\\ Y_{ob,\{0,L,T-L+1\}}\end{array}\right]g, \label{Eq:SeqGen}
			\end{align}
			with $g$ given by
			\begin{align}
			&g=\,\mathcal{M}^{\dagger}_{ob,\{0,L,T-L+1\}}\left[\begin{array}{c}\nu_{[0,L]}\\ \nu^2_{[0,L]}\end{array}\right]\nonumber\\
			&+\left(\mathbf{I}_{T-L+1}-\mathcal{M}^{\dagger}_{ob,\{0,L,T-L+1\}}\mathcal{M}_{ob,\{0,L,T-L+1\}}\right)w,\label{Eq:Gdefinition}
			\end{align}
			where $\mathcal{M}_{ob,\{0,1,T\}}$ is given in \eqref{Eq:HankelMcal} and $w\in\mathbb{R}^{T-L+1}$ is an arbitrary vector, correspond to trajectories of the system \eqref{Eq:VolterraEq02}.
		\end{enumerate}
	\end{theorem}
	
	The claims in Theorem \ref{Theo:Representation} exhibit two important differences with respect to the related results for linear systems \cite{Willems2005,Persis2019}. First, Theorem \ref{Theo:Representation}, part ii) establishes that not every linear combination of the input-output data contained in the Hankel matrices $M_{ob,\{0,L,T-L+1\}}$, $M^2_{ob,\{0,L,T-L+1\}}$ and $Y_{ob,\{0,L,T-L+1\}}$, actually corresponds to a trajectory $\{\mu_{[0,L]},\mu^2_{[0,L]},y_{[0,L]}\}$ of the system \eqref{Eq:VolterraEq02}. Second, every trajectory $\{\mu_{[0,L]},\mu^2_{[0,L]},y_{[0,L]}\}$ of the system can be represented as a linear combination of the Hankel matrices $M_{ob,\{0,L,T-L+1\}}$, $M^2_{ob,\{0,L,T-L+1\}}$ and $Y_{ob,\{0,L,T-L+1\}}$, a characteristic that is reflected in Theorem \ref{Theo:Representation}, part i). 
	
	
	Now, as a direct consequence of Theorem \ref{Theo:Representation}, and for $L=1$, a result analogous to \cite[Theo.~1]{Persis2019} is presented. It enables the replacement of the model-based representation of the system in \eqref{Eq:VolterraEq01} or \eqref{Eq:VolterraEq02} by data.
	
	\begin{corollary}[\textbf{Data-based representation of a second-order DVS}]
		\label{Cor:DataRepresentation}
		Consider the system \eqref{Eq:VolterraEq02} and observed input-output sequences $u_{ob,[-M,T]}$ and $y_{ob,[0,T]}$. Iff $u_{ob,[-M,T]}$ is persistently exciting of order $L=1$ w.r.t. the system \eqref{Eq:VolterraEq02}, then the system \eqref{Eq:VolterraEq02} can be equivalently represented through data as
		\begin{align}
		\begin{split}
		\mu(k)&=\left[\begin{array}{c c c c}u(k) & u(k-1) & \cdots & u(k-M)\end{array}\right]^\top,\\
		y(k)&=Y_{ob,\{0,1,T\}}\mathcal{M}^{\dagger}_{ob,\{0,1,T\}}\left[\begin{array}{c}\mu(k)\\ \mu^2(k)\end{array}\right],
		\end{split} \label{Eq:VolterraData01}
		\end{align}
		where $\mathcal{M}_{ob,\{0,1,T\}}$ is defined in \eqref{Eq:HankelMcal} and $Y_{ob,\{0,1,T\}}$ is the associated Hankel matrix to $y_{ob,[0,T]}$.
	\end{corollary}
	
	The two results of this section, Theorem \ref{Theo:Representation} and Corollary~\ref{Cor:DataRepresentation}, are analogous to the classical representation theorems for linear systems. Theorem \ref{Theo:Representation} provides a representation of the input-output trajectories of the system \eqref{Eq:VolterraEq02} and, in this respect, describes the \emph{behavior} of the system. Corollary \ref{Cor:DataRepresentation}, on the other hand, is intended to represent the input-output behavior via data, thus, to replace the system model. It is also important to remark that the excitation condition required in Corollary~\ref{Cor:DataRepresentation} corresponds to classical requirements for the parameter identification in this class of systems \cite[Sec.~4.1]{Doyle2002}.

	For the data-driven control design tackled in the next section, it is convenient to independently represent the linear and nonlinear operators $\mathbf{P}_1$ and $\mathbf{P}_2$, respectively. This is achieved via the following specialization of Corollary \ref{Cor:DataRepresentation}, which can be obtained by splitting $\mathcal{M}_{ob,\{0,1,T\}}$ in \eqref{Eq:VolterraData01} in two separate matrices. 
	\begin{corollary}[\textbf{Data-based representation of $\mathbf{P}_1$ and $\mathbf{P}_2$}]
		\label{Cor:DataRepresentation2}	
		Consider the system \eqref{Eq:VolterraEq02} and observed input-output sequences $u_{ob,[-M,T]}$ and $y_{ob,[0,T]}$. Iff $u_{ob,[-M,T]}$ is persistently exciting of order $L=1$ w.r.t. the system \eqref{Eq:VolterraEq02}, then the operators $\mathbf{P}_1$ and $\mathbf{P}_2$ in \eqref{Eq:VolterraEq02} can be represented through data as
		\begin{align}
		\begin{split}
		\mu(k)&=\left[\begin{array}{c c c c}u(k) & u(k-1) & \cdots & u(k-M)\end{array}\right]^\top,\\
		\mathbf{P}_1\big(u(k)\big)&=P_1\,\mu(k),\\
		\mathbf{P}_2\big(u(k)\big)&=P_2\,\mu^2(k),
		\end{split} \label{Eq:VolterraData02}
		\end{align}
		where the matrices $P_1$ and $P_2$ are defined in \eqref{Eq:OpDataBased}.
	\end{corollary}

\section{DATA-DRIVEN FORMULA FOR INTERNAL MODEL CONTROL OF DVSs}

	\begin{figure*}[t!]
		\begin{small}
			\begin{align}
			\begin{split}
			P_1\!&=\!Y_{ob,\{0,1,T\}}\!\left(\mathbf{I}_{T}-\left(M^2_{ob,\{0,1,T\}}\right)^{\dagger}M^2_{ob,\{0,1,T\}}\right)M^{\dagger}_{ob,\{0,1,T\}}\left(\mathbf{I}_{M+1}\!-\!M_{ob,\{0,1,T\}}\left(M^2_{ob,\{0,1,T\}}\right)^{\dagger}M^2_{ob,\{0,1,T\}}M^{\dagger}_{ob,\{0,1,T\}}\right)^{\!-1},\\
			P_2\!&=\!Y_{ob,\{0,1,T\}}\!\left(\mathbf{I}_{T}-M^{\dagger}_{ob,\{0,1,T\}}M_{ob,\{0,1,T\}}\right)\!\left(M^2_{ob,\{0,1,T\}}\right)^{\dagger}\!\left(\mathbf{I}_{\frac{1}{2}(M+1)(M+2)}\!-\!M^2_{ob,\{0,1,T\}}M^{\dagger}_{ob,\{0,1,T\}}M_{ob,\{0,1,T\}}\left(M^2_{ob,\{0,1,T\}}\right)^{\dagger}\right)^{\!-1}
			\end{split}\label{Eq:OpDataBased}			
			\end{align}			
		\end{small}
	\end{figure*}
	
	\label{Sec:DataDrivenCon}
	\begin{figure}[t!]
		\centering
		\includegraphics[scale=0.6,draft=false]{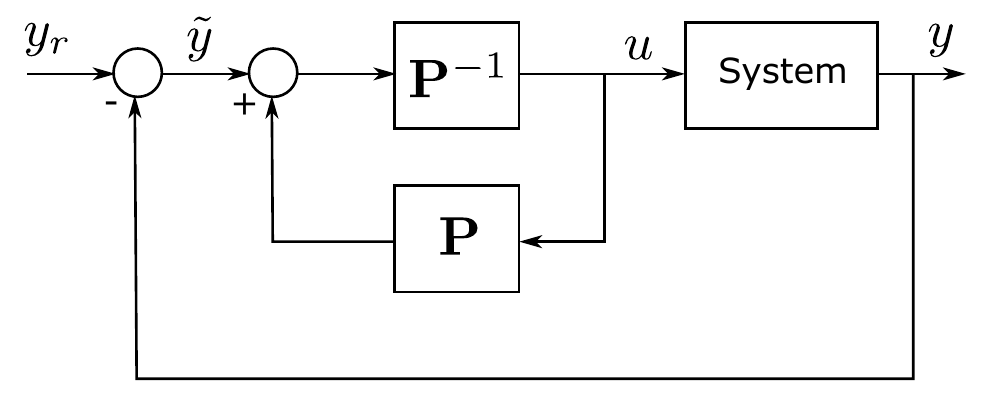}
		\caption{Basic structure of the Internal Model Control (IMC) with reference $y_r$ and tracking error $\tilde y$, based on \cite[Sec.~6]{Doyle2002}.}
		\label{Fig:IMC}
	\end{figure}
	\begin{figure}[t!]
		\centering
		\includegraphics[scale=0.6,draft=false]{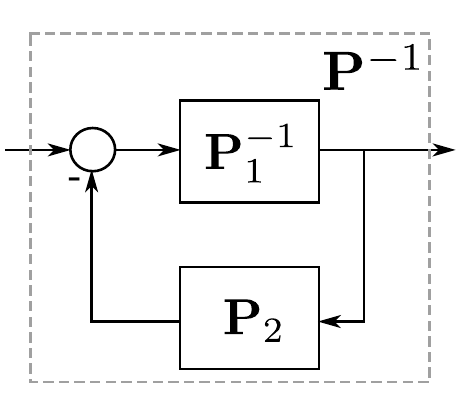}
		\caption{Implementation of the system inverse $\mathbf{P}^{-1}$ using the nonlinear operator $\mathbf{P}_2$ and the inverse of the linear operator $\mathbf{P}_1$, both defined in \eqref{Eq:OpDef}.}
		\label{Fig:PInv}
	\end{figure}
	
Next, the results of the previous section are specialized in order to be applied for output-tracking of second-order DVSs by means of a data-driven IMC algorithm. 
As the name suggests, IMCs are model-based and, essentially, rely on inverting the system input-output map. More precisely, as illustrated in Figure \ref{Fig:IMC}, the fundamental idea of the IMC principle is as follows \cite{Doyle1995,LingRiv1998,Doyle2002}. Given a reference trajectory $y_r:\mathbb{Z} \to \mathbb{R}$, use the system model represented by the operator $\mathbf{P}=\mathbf{P}_1+\mathbf{P}_2$ (see \eqref{Eq:OpDef}) to produce a feed-forward term and add it to the tracking error $\tilde y$ in order to cancel the known part of the system response. Then, the inverse of the system model $\mathbf{P}^{-1}$ is employed to produce the input that, in the absence of uncertainties, will achieve the desired output. 
	
	Hence, given a reference trajectory $y_r:\mathbb{Z} \to \mathbb{R}$ for the system \eqref{Eq:VolterraEq02}, the corresponding IMC algorithm reads \cite{LingRiv1998,Doyle2002}
	\begin{equation}
		u=\mathbf{P}^{-1}\big(\mathbf{P}(u)-y+y_r\big).
		\label{imc2}
	\end{equation} 
	To derive a data-driven IMC formula for the second-order DVS \eqref{Eq:VolterraEq02}, we recall that $\mathbf{P_1}$ in \eqref{Eq:VolterraEq02}  is a linear causal stable operator and make the following two standard assumptions \cite[Sec.~6.2]{Doyle2002}.
	\begin{assumption}
		$\phantom{A}$
		\begin{enumerate}
			\item[i)]	The linear operator $\mathbf{P_1}$ in \eqref{Eq:VolterraEq02} has minimum phase, i.e., all its zeros lay strictly inside the unit circle. 
			\item[ii)]	The inverse of the operator $\mathbf{I}+\mathbf{P}^{-1}_1\mathbf{P}_2$ exists, where $\mathbf{I}$ corresponds to the identity operator.
		\end{enumerate}
		\label{ass1}
	\end{assumption}

	Under Assumption~\ref{ass1}, the inverse operator $\mathbf{P}^{-1}_1$ is stable, and the inverse of the whole system $\mathbf{P}^{-1}$ can be implemented following Figure \ref{Fig:PInv}, see \cite[Sec.~6]{Doyle2002}. It follows that, compared to the general IMC structure, for the implementation of \eqref{imc2} only $\mathbf{P}^{-1}_1$ is needed and not the inverse of the nonlinear operator $\mathbf{P}_2$.
	
	To streamline our main result, the following preliminary lemma is helpful, in which a data-based representation of the inverse operator $\mathbf{P}^{-1}_1$ is provided. 	
	This requires to introduce the part of the observed output data corresponding to $\mathbf{P}_1$, which we denote by $y_{1ob,[0,T]}$. 
	By using Corollary~\ref{Cor:DataRepresentation2}, $y_{1ob,[0,T]}$ can be extracted following any of the next two expressions:
	\begin{align}
	\begin{split}
		y_{1ob,[0,T]}&=\left(\mathbf{I}_{T}\otimes P_1\right)\mu_{ob,[0,T]}\\
			&=y_{ob,[0,T]}-\left(\mathbf{I}_{T}\otimes P_2\right)\mu^2_{ob,[0,T]},
		\end{split}\label{Eq:y1ob}
	\end{align}
	with $P_1$ and $P_2$ as in \eqref{Eq:OpDataBased}. Let $u(k)$ be the output of $\mathbf{P}^{-1}_1$ when applied to $y_1(k)$, i.e., $u(k)=\mathbf{P}^{-1}_1(y_1(k))$. The following result characterizes $\mathbf{P}^{-1}_1$ through data. Its proof is given in the Appendix. 
	
	\begin{lemma}[\textbf{Data-based representation of $\mathbf{P}^{-1}_1$}]
		\label{Lem:Inverse}
		Consider the system \eqref{Eq:VolterraEq02} and observed input-output sequences $u_{ob,[-M,T]}$ and $y_{ob,[0,T]}$. Assume that $u_{ob,[-M,T]}$ is persistently exciting of order $L=1$ w.r.t. to the system \eqref{Eq:VolterraEq02}. Build $y_{1ob,[0,T]}$ following \eqref{Eq:y1ob}. Iff $y_{1ob,[0,T]}$ is persistently exciting of order one in the linear sense, i.e., the Hankel matrix $Y_{1ob,\{0,1,T\}}$ has full-row rank, then the inverse operator $\mathbf{P}^{-1}_1$ can be represented through data as
		\begin{align}
		\begin{split}
		u(k)&=U_{ob,\{0,1,T\}}\left[\begin{array}{c}Y_{1ob,\{0,1,T\}}\\ \hline X_{ob,\{0,1,T\}}\end{array}\right]^{\dagger}\left[\begin{array}{c}y(k)\\ \chi(k)\end{array}\right],\\
		\chi(k)&=\left[\begin{array}{c c c c}u(k-1) & u(k-2) & \cdots & u(k-M)\end{array}\right]^\top,
		\end{split}\label{Eq:InverseSys}
		\end{align}
		where $U_{ob,\{0,1,T\}}$, $X_{ob,\{0,1,T\}}$ and $Y_{1ob,\{0,1,T\}}$ are the Hankel matrices associated to $u_{ob,[0,T]}$, $\chi_{ob,[0,T]}$ and $y_{1ob,[0,T]}$, respectively.
	\end{lemma}

	It is important to emphasize that $\chi(k)$ in \eqref{Eq:InverseSys} stores the past outputs and plays the role of internal state of the inverse operator $\mathbf{P}^{-1}_1$. We are now in the position to state our main result. Recall the matrices $U_{ob,\{0,1,T\}},$ $Y_{1ob,\{0,1,T\}}$ and $X_{ob,\{0,1,T\}}$ defined in Lemma~\ref{Lem:Inverse} as well as the matrix $P_1$ defined in \eqref{Eq:OpDataBased}.
	
	\begin{theorem}[\textbf{Data-driven IMC formula for output tracking}]
		\label{The:main}
		Consider the system \eqref{Eq:VolterraEq02} with Assumption~\ref{ass1} and observed input-output sequences $u_{ob,[-M,T]}$ and $y_{ob,[0,T]}$. 
		Let $y_r:\mathbb{Z} \to \mathbb{R}$ be a desired reference trajectory.
		Iff $u_{ob,[-M,T]}$ is persistently exciting of order $L=1$ w.r.t. the system \eqref{Eq:VolterraEq02}, then the IMC algorithm in \eqref{imc2} corresponding to the block diagrams in Figures \ref{Fig:IMC} and \ref{Fig:PInv} can be implemented through the observed data as
		\begin{equation}
			\begin{split} 
				u(k)&=U_{ob,\{0,1,T\}}\left[\begin{array}{c}Y_{1ob,\{0,1,T\}}\\ \hline X_{ob,\{0,1,T\}}\end{array}\right]^{\dagger}\left[\begin{array}{c}P_1\mu(k)+\tilde{y}(k)\\ \chi(k)\end{array}\right],\\
				\chi(k)&=\left[\begin{array}{c c c c}u(k-1) & u(k-2) & \cdots & u(k-M)\end{array}\right]^\top,\\
				\mu(k) &= \left[\begin{array}{c c c c}u(k) & u(k-1) & \cdots & u(k-M)\end{array}\right]^\top,\\
				\tilde{y}(k)&=y_{r}(k)-y(k).
			\end{split}
			\label{Eq:ControlSys}
		\end{equation}
	\end{theorem}

	\begin{proof}
		Consider the block diagram in Figure \ref{Fig:IMC} and replace $\mathbf{P}^{-1}$ by its implementation in Figure \ref{Fig:PInv}. By reducing the block diagram at the summation point and noting that $\mathbf{P}(u)-\mathbf{P}_2(u)=\mathbf{P}_1(u)$, the input $u$ results in
		\begin{align}
			u=\mathbf{P}^{-1}_1\big(\mathbf{P}_1(u)-y+y_r\big)=\mathbf{P}^{-1}_1\big(\mathbf{P}_1(u)+\tilde y\big).
			\label{Eq:control}
		\end{align}
		Thus, for the implementation of \eqref{Eq:control} only the linear operator $\mathbf{P}_1$ and its inverse are needed. The data-based representations of the operators $\mathbf{P}_1$ and $\mathbf{P}^{-1}_1$ are obtained from Corollary \ref{Cor:DataRepresentation2} and Lemma \ref{Lem:Inverse}, respectively. The combination of expressions \eqref{Eq:VolterraData02} and \eqref{Eq:InverseSys} yields \eqref{Eq:ControlSys}, i.e., the equivalent data-driven formula to \eqref{Eq:control}.
	 	\end{proof}	
	
	The standard approach to implement the IMC \eqref{imc2} is to pursue a two-step design procedure: In the first step, the parameters $\alpha_i$ and $\beta_{(i,j)}$ in \eqref{Eq:VolterraEq01} are identified. Then, in the second step the inverse of the linear part is computed using, e.g., the $Z$-transform and the IMC algorithm is implemented.
	In contrast, Theorem~\ref{The:main} follows the spirit of data-driven control \cite{DataDriven2012,HouWang2013,Persis2019} by providing the concise formula \eqref{Eq:ControlSys}, that enables a direct implementation of the IMC \eqref{imc2} from data in a single step.

\section{SIMULATION EXAMPLE}
	\label{Sec:Simulation}
	
	\begin{figure}[t]
		\centering
		\includegraphics[width=0.4\textwidth]{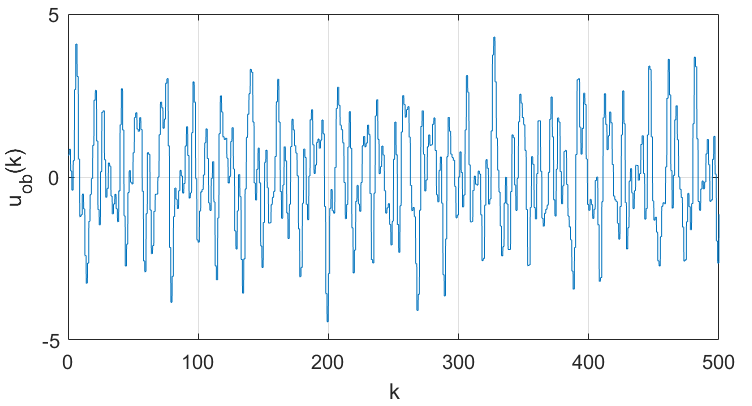}
		\caption{Example 1: Input signal used to characterize the second-order DVS.}\label{Fig:ExcSig01}
	\end{figure}
	
	\begin{figure}[t]
		\centering
		\includegraphics[width=0.4\textwidth]{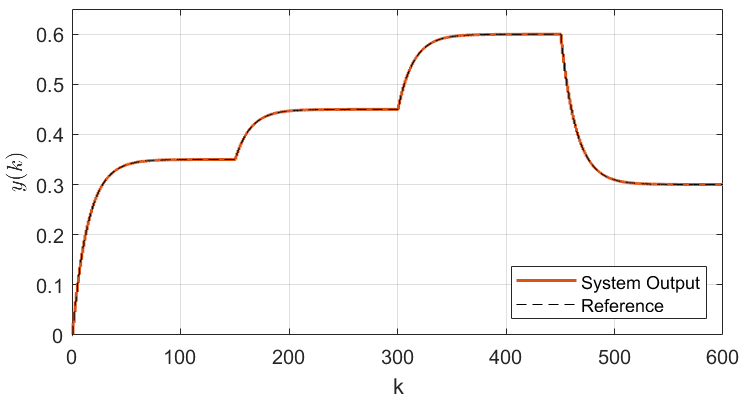}
		\caption{Example 1: By using the proposed data-driven IMC formula \eqref{Eq:ControlSys}, perfect reference tracking of the second-order DVS is achieved.}\label{Fig:Track01}
	\end{figure}
	
	To illustrate the application of the proposed data-driven IMC formula \eqref{Eq:ControlSys}, two scenarios are considered. At first, we investigate the control performance for an ideal second-order DVS of the form \eqref{Eq:VolterraEq02}. In the second example, we then evaluate the robustness of the derived methodology via its application to a continuous-time bilinear system, which does not correspond exactly to a second-order DVS. In addition, the impact of noise on the control performance is investigated. 
	
	\noindent
	\textbf{Example 1.} For the first scenario, a second-order DVS \eqref{Eq:VolterraEq02}, \eqref{Eq:OpDef} of dimension $M=5$ is considered with the parameters
	\begin{align*}
	\theta_1^\top&=\left[\begin{array}{c c c c c c} 4 & 3 & 0.82 & 0.156 & -0.014 & -0.006\end{array}\right],\\
	\theta_2^\top&=\left[\begin{array}{c c c c c c} 0.8147 & 0.9058 & 0.127 & 0.9134 & 0.6324 & 0.0975\end{array}\right.\\
	&\quad\begin{array}{c c c c c c} 0.2785 & 0.5469 & 0.9575 & 0.9649 & 0.1576 & 0.9706\end{array}\\
	&\quad\begin{array}{c c c c c c} 0.9572 & 0.4854 & 0.8003 & 0.1419 & 0.4218 & 0.9157\end{array}\\
	&\quad\left.\begin{array}{c c c} 0.7922 & 0.9595 & 0.6557\end{array}\right].
	\end{align*}
	The parameters are chosen such that the conditions in Assumption~\ref{ass1} are satisfied. The system is excited with a superposition of discrete sinusoidal signals, shown in Figure~\ref{Fig:ExcSig01}. After verifying that the signal is persistently exciting, see Definition~\ref{Def:PE}, the collected input-output data is used to compute the IMC using directly the formula \eqref{Eq:ControlSys}. The output reference for the tracking process is shown in Figure~\ref{Fig:Track01}. As can be seen, perfect reference tracking is achieved. This is possible because the considered system coincides exactly with a second-order DVS and the use of the inverse model in the IMC imposes this behavior.
	
	\noindent
	\textbf{Example 2.} Now, the following bilinear system is considered
	\begin{align}
		\begin{split}
			\dot{x}_1(t)&=x_2(t)+x_2(t)u(t),\\
			\dot{x}_2(t)&=-x_1(t)-4x_2(t)-x_1(t)u(t)+u(t),\\
			y(t)&=x_1(t)-x_2(t).
		\end{split}\label{Eq:BilinearSys}
	\end{align}

	It is known that the response of bilinear systems can be approximated by a DVS, though in general the DVS order is significantly higher than two \cite[Sec.~4.7]{Doyle2002}, \cite{Rugh2002}. Hence, this represents a non ideal scenario for our derived results.
	
	The system \eqref{Eq:BilinearSys} is simulated using a fourth-order Runge-Kutta method with a step size of $1\times 10^{-4}$ [s]. At the input and output of the system \eqref{Eq:BilinearSys}, a zero-order hold with a sampling time of $1.5$ [s] is introduced in order to obtain sampled input-output data. 	 
	Additionally, uniform noise in the interval $[-0.05,0.05]$ (which corresponds to approximately $20\%$ of the maximum output signal amplitude) is added to the measured output. As in Example~1, the data is generated using a superposition of discrete sinusoidal input signals. The system initial conditions for this process are $x_1(0)=1$ and $x_2(0)=0$. The observed input-output data is shown in Figure~\ref{Fig:ExcSig02}.
	
	To implement a data-driven output-tracking IMC for the system \eqref{Eq:BilinearSys}, its dynamics is approximated by the second-order DVS \eqref{Eq:VolterraEq02}, \eqref{Eq:OpDef}. Then the controller formula \eqref{Eq:ControlSys} is used. Thereby, the key design parameter is the dimension $M$, which determines the number of parameters of the second-order DVS, see \eqref{Eq:OpDef}. It follows from Definition~\ref{Def:PE} that there is an inherent trade-off between the magnitude of $M$ and the amount of required data needed to verify persistence of excitation. To evaluate the influence of $M$ on both the data requirements and the control performance, we pursue three IMC designs with $M=3$, $M=4$ and $M=5$.
	
	In the present case, we can verify that the generated input signal is persistently exciting for all our three choices of $M$. Hence, we proceed to the IMC design via \eqref{Eq:ControlSys}. The resulting tracking performance for zero initial conditions and a smooth reference signal is shown in Figure \ref{Fig:Track02}. The corresponding tracking errors are given in Figure \ref{Fig:TrackErr}. Since the bilinear system \eqref{Eq:BilinearSys} is not a second-order DVS, exact reference tracking is not achieved. However, all three controllers are capable of tracking the reference with acceptable accuracy, showing that the approach can be used even when the system does not exactly corresponds to a DVS. However, it is clear from Figure \ref{Fig:TrackErr} that the controller for $M=5$ has higher error peaks and exhibits oscillations. This is an indicative of overfitting and thus, it seems recommendable to use the controller with the lowest dimension. 
	
	To robustify the IMC against model uncertainties and external disturbances, it is conventional to add a filter in between the error signal and the system inverse \cite{Doyle2002}. Furthermore, to deal with non-minimum phase systems, an inner-outer factorization of the linear part is required. We are currently exploring, how to incorporate such approaches in a data-driven setting.


	\begin{figure}[t]
		\centering
		\includegraphics[width=0.4\textwidth]{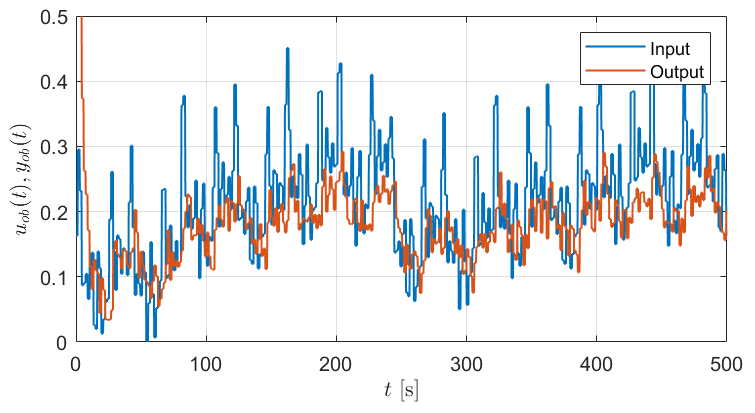}
		\caption{Example 2: Observed input-output data from the bilinear system \eqref{Eq:BilinearSys}.}\label{Fig:ExcSig02}
	\end{figure}
	
	\begin{figure}[t]
		\centering
		\includegraphics[width=0.4\textwidth]{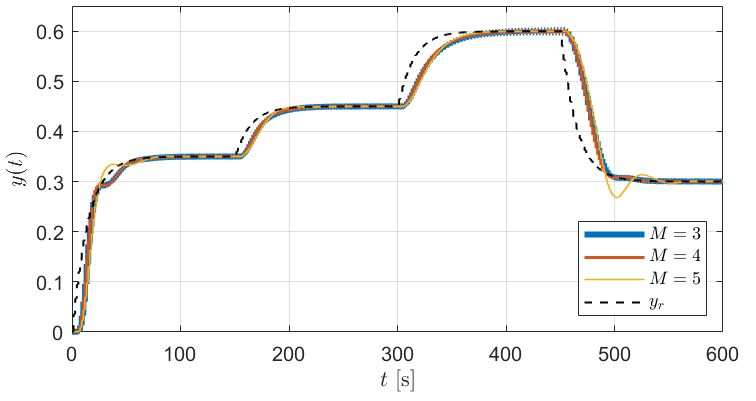}
		\caption{Example 2: Output reference tracking with the data-driven IMC \eqref{Eq:ControlSys} using a second-order DVS of dimension $M=3$, $M=4$ and $M=5$, respectively.}\label{Fig:Track02}
	\end{figure}
	
	
	\begin{figure}[t]
		\centering
		\includegraphics[width=0.4\textwidth]{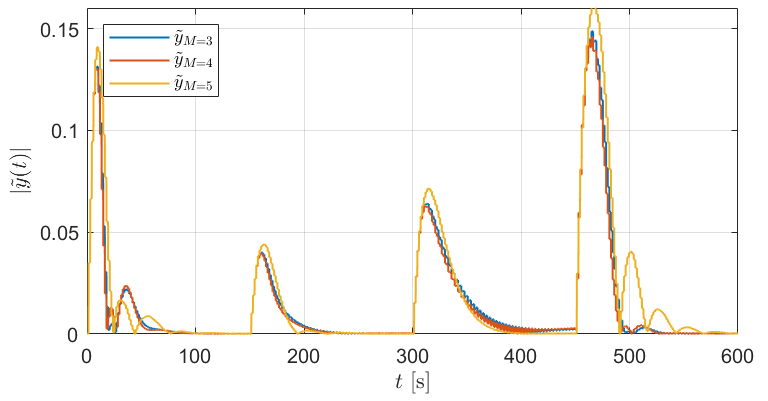}
		\caption{Example 2: Tracking error of the the data-driven IMC \eqref{Eq:ControlSys} using a second-order DVS of dimension $M=3$, $M=4$ and $M=5$, respectively.}\label{Fig:TrackErr}
	\end{figure}

\section{CONCLUSIONS}
	In this work, the fundamental result of Willems and collaborators on data-based representation of the input-output
	behavior of linear systems has been extended to an important class of nonlinear systems, namely second-order DVSs. 
	That is, we have provided an explicit data-dependent representation of a second-order DVS requiring only finite, but rich enough, data. 
	
	By using this new characterization, a \emph{data-driven internal model control} for output-tracking of this class of systems has been proposed. Compared to the conventional two-step ``system identification/control design" procedure, we provide a single formula, in which the controller itself is directly parametrized through data. 
	The performance and robustness of this approach have been illustrated by controlling a bilinear system with noisy measurements.
	
	The methodology can be extended to DVSs of arbitrary order, work that is under development. Additionally, it is under investigation how to address non-minimum phase systems and how to attenuate model mismatches, two problems that frequently appear when the methodology is applied to nonlinear systems that do not correspond exactly to DVSs.




\section*{APPENDIX}

	\begin{figure*}[t]
		\begin{small}
			\begin{align}
			\begin{split}
			D_1&=\left(I_{T}-(M^2_{ob,\{0,1,t\}})^{\dagger}M^2_{ob,\{0,1,T\}}\right)M^{\dagger}_{ob,\{0,1,T\}}\left(I_{M+1}-M_{ob,\{0,1,T\}}(M^2_{ob,\{0,1,T\}})^{\dagger}M^2_{ob,\{0,1,T\}}M^{\dagger}_{ob,\{0,1,T\}}\right)^{-1},\\
			D_2&=\left(I_{T}-M^{\dagger}_{ob,\{0,1,t\}})M_{ob,\{0,1,T\}}\right)(M^2_{ob,\{0,1,T\}})^{\dagger}\left(I_{\frac{M+1}{2}(M+2)}-M^2_{ob,\{0,1,T\}}M^{\dagger}_{ob,\{0,1,T\}}M_{ob,\{0,1,T\}}(M^2_{ob,\{0,1,T\}})^{\dagger}\right)^{-1}.
			\end{split}\label{Eq:D1D2}
			\end{align}
		\end{small}
	\end{figure*}

	\subsection{Proof of Theorem \ref{Theo:Representation}}
		\textbf{Part i)} Define $\sigma_1=L(M+1)$ and $\sigma_2=\frac{1}{2}L(M+1)(M+2)$. Any $L$-long sequences $\mu_{[0,L]}$, $\mu^2_{[0,L]}$ and $y_{[0,L]}$ corresponding to the system \eqref{Eq:VolterraEq02} can be represented as
		\begin{align}
		\left[\begin{array}{c}\mu_{[0,L]}\\ \mu^2_{[0,L]}\\y_{[0,L]}\end{array}\right]=\left[\begin{array}{c c}\mathbf{I}_{\sigma_1} & 0_{\sigma_1\times\sigma_2}\\ 0_{\sigma_2\times\sigma_1} & \mathbf{I}_{\sigma_2}\\\mathbf{I}_{L}\otimes\theta^\top_1 & \mathbf{I}_{L}\otimes\theta^\top_2\end{array}\right]\left[\begin{array}{c}\mu_{[0,L]}\\ \mu^2_{[0,L]}\end{array}\right], \label{Eq:DBD01}
		\end{align}
		with $\theta_1$ and $\theta_2$ as in \eqref{Eq:OpDef}. If $u_{ob,[-M,T]}$ is persistently exciting of order $L$, i.e., the rank condition in Definition \ref{Def:PE} holds, then it follows that
		\begin{align}
		\mathrm{rank}\big(&\mathcal{M}_{ob,\{0,L,T-L+1\}}\big)=\nonumber\\
		&\mathrm{rank}\left(\left[\begin{array}{c | c}\mathcal{M}_{ob,\{0,L,T-L+1\}} & \begin{bmatrix} \mu_{[0,L]}\\ \mu^2_{[0,L]}\end{bmatrix}\end{array}\right]\right)\nonumber\\
		&=\frac{1}{2}L\big(M+1\big)\big(M+4\big).\label{Eq:DBD01rank}
		\end{align}
		Since the rank does not change when the sequences of $\mu_{[0,L]}$ and $\mu^2_{[0,L]}$ are incorporated into $\mathcal{M}_{ob,\{0,1,T\}}$, there exists a $g\in\mathbb{R}^{T-L+1}$ such that
		\begin{align}
		\left[\begin{array}{c}\mu_{[0,L]}\\ \mu^2_{[0,L]}\end{array}\right]=\mathcal{M}_{ob,\{0,L,T-L+1\}}g.\label{Eq:DBD02}
		\end{align}
		By replacing \eqref{Eq:DBD02} in \eqref{Eq:DBD01}, one obtains
		\begin{align}
		\left[\begin{array}{c}\mu_{[0,L]}\\ \mu^2_{[0,L]}\\y_{[0,L]}\end{array}\right]&=\left[\begin{array}{c c}\mathbf{I}_{\sigma_1} & 0_{\sigma_1\times\sigma_2}\\ 0_{\sigma_2\times\sigma_1} & \mathbf{I}_{\sigma_2}\\\mathbf{I}_{L}\otimes\theta^\top_1 & \mathbf{I}_{L}\otimes\theta^\top_2\end{array}\right]\mathcal{M}_{ob,\{0,L,T-L+1\}}g \label{Eq:DBD03}\\ 
		&=\left[\begin{array}{c}M_{ob,\{0,L,T-L+1\}} \\ M^2_{ob,\{0,L,T-L+1\}}\\Y_{ob,\{0,L,T-L+1\}}\end{array}\right]g. \label{Eq:DBD04}
		\end{align}	 
		In this manner, the result of the first part is recovered. 
		
		The necessity stems from the fact that if $\mathcal{M}_{ob,\{0,L,T-L+1\}}$ does not have full-row rank, the existence of $g$ cannot be ensured in general since the rank of \eqref{Eq:DBD01rank} may change.
		
		\textbf{Part ii)} From \eqref{Eq:DBD01}, one has
		\begin{align}
		\left[\begin{array}{c}\mu_{[0,L]}\\ \mu^2_{[0,L]}\\y_{[0,L]}\end{array}\right]&=\left[\begin{array}{c}M_{ob,\{0,L,T-L+1\}} \\ M^2_{ob,\{0,L,T-L+1\}}\\Y_{ob,\{0,L,T-L+1\}}\end{array}\right]g\\
		&=\left[\begin{array}{c c}\mathbf{I}_{\sigma_1} & 0_{\sigma_1\times\sigma_2}\\ 0_{\sigma_2\times\sigma_1} & \mathbf{I}_{\sigma_2}\\\mathbf{I}_{L}\otimes\theta^\top_1 & \mathbf{I}_{L}\otimes\theta^\top_2\end{array}\right]\mathcal{M}_{ob,\{0,L,T-L+1\}}g.
		\end{align}
		By direct substitution of $g$ as defined in \eqref{Eq:Gdefinition}, it follows that
		\begin{align}
		\left[\begin{array}{c}\mu_{[0,L]}\\ \mu^2_{[0,L]}\\y_{[0,L]}\end{array}\right]&=\left[\begin{array}{c c}\mathbf{I}_{\sigma_1} & 0_{\sigma_1\times\sigma_2}\\ 0_{\sigma_2\times\sigma_1} & \mathbf{I}_{\sigma_2}\\\mathbf{I}_{L}\otimes\theta^\top_1 & \mathbf{I}_{L}\otimes\theta^\top_2\end{array}\right]\begin{bmatrix}\nu_{[0,L-1]}\\ \nu^2_{[0,L-1]}\end{bmatrix},
		\end{align}	
		thus, corroborating that the trajectory \eqref{Eq:SeqGen} corresponds to the system \eqref{Eq:VolterraEq02}. The necessity follows from the fact that if $\mathcal{M}_{ob,\{0,L,T-L+1\}}$ does not have full-row rank, then $\mathcal{M}_{ob,\{0,L,T-L+1\}}\mathcal{M}^{\dagger}_{ob,\{0,L,T-L+1\}}\neq \mathbf{I}_{L(\sigma_1+\sigma_2)}$. Therefore, in general, the trajectory \eqref{Eq:SeqGen} will no correspond to the system \eqref{Eq:VolterraEq02}.
		
	\subsection{Proof of Corollary \ref{Cor:DataRepresentation}}
		If the observed input sequence $u_{ob,[-M,T]}$ is persistently exciting of order one with respect to the system \eqref{Eq:VolterraEq02}, $\mathcal{M}_{ob,\{0,1,T\}}$ has full-row rank. Then Theorem \ref{Theo:Representation} part i) implies that there exists a $g\in\mathbb{R}^{T},$ such that
		\begin{align*}
		y(k)=Y_{ob,\{0,1,T\}}g.
		\end{align*}
		The proof is completed by choosing $g$ as in Theorem \ref{Theo:Representation} part ii) with $\nu_{[0,1]}=\mu(k)$, which yields \eqref{Eq:VolterraData01}.
		
	\subsection{Proof of Corollary \ref{Cor:DataRepresentation2}}
		Define $\sigma_1=M+1$ and $\sigma_2=\frac{1}{2}(M+1)(M+2)$. One has that
		\begin{align}
		\mathcal{M}_{ob,\{0,1,T\}}\mathcal{M}^{\dagger}_{ob,\{0,1,T\}}&=\left[\begin{array}{c}M_{ob,\{0,1,T\}}\\ M^2_{ob,\{0,1,T\}}\end{array}\right]\mathcal{M}^{\dagger}_{ob,\{0,1,T\}}\nonumber\\
		&=\mathbf{I}_{\sigma_1+\sigma_2}. \label{minv}
		\end{align}
		Let $\mathcal{M}^{\dagger}_{ob,\{0,1,T\}}$ be split in two matrices $D_1$ and $D_2$, i.e., $\mathcal{M}^{\dagger}_{ob,\{0,1,T\}}=[D_{1} \; D_{2}]$. Then, \eqref{minv} can be written equivalently as
		\begin{align*}
		\left[\begin{array}{c}M_{ob,\{0,1,T\}}\\ M^2_{ob,\{0,1,T\}}\end{array}\right]&\left[\begin{array}{c c}D_1 & D_2\end{array}\right]=\nonumber\\
		=&\left[\begin{array}{c c}M_{ob,\{0,1,T\}}D_1 & M_{ob,\{0,1,T\}}D_2\\ M^2_{ob,\{0,1,T\}}D_1 & M^2_{ob,\{0,1,T\}}D_2\end{array}\right]\\
		=&\left[\begin{array}{c c}\mathbf{I}_{\sigma_1} & 0_{\sigma_1\times\sigma_2}\\ 0_{\sigma_2\times\sigma_1} & \mathbf{I}_{\sigma_2}\end{array}\right].
		\end{align*}
		Solving for $D_1$ and $D_2$ yields \eqref{Eq:D1D2}.
		
		By substituting the relation $\mathcal{M}^{\dagger}_{ob,\{0,1,T\}}=[D_1\;D_2]$ in \eqref{Eq:VolterraData01}, we obtain
		\begin{align*}
		y(k)&=Y_{ob,\{0,1,T\}}D_1\mu(k)+Y_{ob,\{0,1,T\}}D_2\mu^2(k).
		\end{align*}
		From the last expression, $P_1$ and $P_2$ in \eqref{Eq:OpDataBased} follow, completing the proof.
		
	\subsection{Proof of Lemma~\ref{Lem:Inverse}}
		From \eqref{Eq:VolterraEq02} it follows that
		\begin{align*}
		y_1(k)=\theta_{1,1}u(k)+\theta_{1,2}u(k-1)+\cdots+\theta_{1,M+1}u(k-M),
		\end{align*}
		where $\theta_{1,i}$ denotes the $i$-th element of $\theta_1$. In order to represent the input $u(k)$ as a function of the output $y(k)$, the relation above can be reorganized as
		\begin{align*}
		u(k)&=\frac{1}{\theta_{1,1}}y_1(k)-\frac{\theta_{1,2}}{\theta_{1,1}}u(k-1)\cdots-\frac{\theta_{1,M+1}}{\theta_{1,1}}u(k-M).
		\end{align*}
		Consider the following short-hands:
		\begin{align*}
		\bar{\theta}=\left[\begin{array}{c c c c}-\frac{\theta_{1,2}}{\theta_{1,1}} & -\frac{\theta_{1,3}}{\theta_{1,1}} & \cdots & -\frac{\theta_{1,M+1}}{\theta_{1,1}}\end{array}\right]^\top, && d=\frac{1}{\theta_{1,1}}.
		\end{align*}
		Hence, the inverse system is described by the linear system
		\begin{align}	
		\begin{split}
		\chi(k)&=\left[\begin{array}{c c c c}u(k-1) & u(k-2) & \cdots & u(k-M)\end{array}\right]\!^\top,\!\\
		u(k)&=\bar{\theta}^\top\chi(k)+d\,y_1(k).
		\end{split}\label{Eq:P1in01}
		\end{align} 
		
		It follows that any $L$-long input-output sequence of \eqref{Eq:P1in01} can be represented as 
		\begin{align}
		\left[\begin{array}{c}y_{1,[0,L]}\\ u_{[0,L]}\end{array}\right]=\left[\begin{array}{c c} \mathbf{I}_{L} & 0_{L\times M}\\ d\,\mathbf{I}_{L} & \mathbf{I}_{L}\otimes\bar{\theta}^\top\end{array}\right]\,\left[\begin{array}{c}y_{1,[0,L]}\\ \chi_{[0,L]}\end{array}\right].\label{Eq:PinvRep00}
		\end{align}
		Since $y_{1ob,[0,T]}$ is persistently exciting, there exists $g\in\mathbb{R}^T$ such that 
		\begin{align}
		\left[\begin{array}{c}y_{1,[0,L]}\\ \chi_{[0,L]}\end{array}\right]&=\left[\begin{array}{c}Y_{1ob,\{0,L,T-L+1\}}\\ X_{ob,\{0,L,T-L+1\}}\end{array}\right]g.\label{Eq:PinvRep01}
		\end{align}
		A particular $g$ can be find using the right inverse.
		By replacing \eqref{Eq:PinvRep01} and \eqref{Eq:PinvRep02} in \eqref{Eq:PinvRep00}, we obtain
		\begin{align*}
		&\left[\begin{array}{c}y_{1,[0,L]}\\ u_{[0,L]}\end{array}\right]=\\
		&\;\left[\begin{array}{c}Y_{1ob,\{0,L,T-L+1\}}\\ U_{ob,\{0,L,T-L+1\}}\end{array}\right]\,\left[\begin{array}{c}Y_{1ob,\{0,L,T-L+1\}}\\ X_{ob,\{0,L,T-L+1\}}\end{array}\right]^{\dagger}\left[\begin{array}{c}y_{1,[0,L]}\\ \chi_{[0,L]}\end{array}\right].
		\end{align*}
		The expression \eqref{Eq:InverseSys} follows by setting $L=1$ above.


\bibliographystyle{IEEEtran}
\bibliography{IEEEabrv,ref_2}

\end{document}